\documentclass{amsproc}
  \usepackage{amsmath}
  \usepackage{graphics}
  \usepackage{epsfig}
  \usepackage[latin1]{inputenc}
  \usepackage[english]{babel}
  \usepackage{amssymb}
  \usepackage{hyperref}
  \usepackage{verbatim}
\newtheorem{theorem}{Theorem}[section]

\theoremstyle{definition}
\newtheorem{definition}[theorem]{Definition}

\newtheorem{example}{Example}

\title[Simple permutations with order $4n + 2$. Part I]{Simple permutations with order $4n + 2$. Part I}


 \email{pacostahumanez@uninorte.edu.co}
 \email{oscare.martinez@usa.edu.co}

\thanks{The first author is supported by the MICIIN/FEDER grant MTM2009-06973 and by Universidad del Norte. The second author is supported by Universidad Sergio Arboleda.}

\begin{document}
\maketitle

%% Enter the first author's name and address:
\centerline{\scshape Primitivo B. Acosta-Hum\'anez}
\medskip
{\footnotesize
 %% please put the address of the first author
 \centerline{Departamento de Matem\'aticas y Estad\'istica}
 \centerline{Universidad del Norte}
 \centerline{Barranquilla, Colombia}
} %% Do not forget to end the {\footnotesize by the sign }

\medskip

\centerline{\scshape Eduardo Mart\'inez Castiblanco}
\medskip
{\footnotesize
 %% please put the address of the second author
 \centerline{Escuela de Matem\'aticas}
 \centerline{Universidad Sergio Arboleda}
 \centerline{Bogot\'a, Colombia}
} %

\bigskip

%% The name of the associate editor will be entered by an editorial staff
% \centerline{(Communicated by the associate editor name)}
\centerline{Dedicated to Jes\'us Hernando P\'erez (Pelusa), teacher and friend.}

	\begin{abstract}
The problem of genealogy of permutations has been solved partially by Stefan (odd order) and Acosta-Hum\'anez \& Bernhardt (power of two). It is well known that Sharkovskii's theorem shows the relationship between the cardinal of the set of periodic points of a continuous map, but simple permutations will show the behavior of those periodic points. This paper studies the structure of permutations of mixed order $4n+2$, its properties and a way to describe its genealogy by using Pasting and Reversing.
	\end{abstract}

	\section{Introduction}
According to the Encyclop\ae dia Brittanica, genealogy is ``the study of family origins and history".  In this paper we are going to show the transposition of this idea to the context of combinatorial dynamics and particulary, how genealogy can be used as a way to understand forcing relationships between periodic points.

		\subsection{Historical Background}
The Combinatorial Dynamics as a field appears in 1964 with the paper ``Co-Existence of Cycles of a Continuous Mapping of a Line onto Itself"  written by Oleksandr Mikolaiovich Sharkovskii (see \cite{Sharkovskii,Sharkovskii2}). From this point and on, the study of algebraic and topological relationships of continuous functions in $\mathbb{R}$ becomes important. In this context, permutations can be used to show minimal orbits.

Subsequently, Louis Block defined a special kind of permutations called simple orbits (also known as Block's Orbits) (see \cite{Block1,Block2}). In this definition Block emphasizes in the three different kinds of orbits related to the three different ``tails" in Sharkovskii's order. The odd order simple permutations are also known as Stefan Orbits (see \cite{Stefan}). There exists two per order and its forcing (genealogy) can be described as two separated lines according to the first permutation.

Chris Bernhardt suggested the study of predecessors and successors of a simple permutation in his paper ``Simple permutations with order a power of two" (see \cite{Bernhardt}) in which he describes a procedure to find the predecessor and the successor of any permutation of order a power of two by using transpositions, rising up a tree with this partial ordering.

Recently the first author gave a new perspective to Bernhardt's results, by including the operations Pasting and Reversing in order to obtain a recursive algorithm that produces the genealogy lines in order a power of two (see \cite{Genealogia}). From this point and on, Pasting and Reversing becomes an important way to study the genealogy problem in the remaining order: Mixed Order.

In this paper we will show a procedure to construct simple permutations by using pasting and reversing. It can be considered as a first step to find the genealogy of mixed order. At this moment, we are interested in the analysis of the algebraic structure of the simple permutations. The dynamical structure will be presented in a further work.

	\subsection{Preliminaries}

The theoretical background used to develop the genealogy's problem is divided in two subsections: the basics of combinatorial dynamics and the Pasting and Reversing operations. In the first part we only state the definitions. In the second part the theorems will be proofed carefully.

	\subsection{Combinatorial Dynamics}
The theoretical background, see \cite{Combinatorial,Minimal,Block1,Ho, Misiurewicz}, starts with Sharkovskii's theorem, then, primitive functions and Markov graphs will be defined.

    \subsubsection{Sharkovskii's Theorem}
The main result of Sharkovskii relates a new order of the natural numbers and the existence of cycles (periodic points). He defined a new order for the set of natural numbers as follows: $n_{1}$ precedes $n_{2}$ ($n_{1} \prec n_{2}$) if for every continuous map of the line into itself the existence of a cycle of order $n_{1}$ implies the existence of a cycle of order $n_{2}$.
		\begin{theorem}[Sharkovskii's Theorem, \cite{Sharkovskii}]
The introduced relationship transforms the set of natural numbers into an ordered set, ordered in the following way:
\begin{equation*}
3 \prec 5 \prec 7 \prec 9 \prec 11 \prec \dots \prec 3\cdot2 \prec 5\cdot2 \prec \dots \prec 3\cdot2^{2} \prec 5\cdot2^{2} \prec \dots \prec 2^{3} \prec 2^{2} \prec 2 \prec 1\end{equation*}
		\end{theorem}

    \subsubsection{Simple Permutations}
$(S_{n},\circ)$ denotes the group of permutations of order ${n}$. These permutations and partitions will be used to describe periodic orbits of continuous functions. A partition of a continuous interval in ${n-1}$ intervals will be defined as
\begin{equation*}
P_{n}=\{x_{i},x_{i+1} \in \mathbb{R} : x_{i} < x_{i+1}, \forall i = 1,\dots,n-1 \}
\end{equation*}

    \begin{definition}[Set of permutations]
A permutation $\theta$ belongs to the set of permutations of $f$ (named $Perm(f)$) if and only if there exists a partition $P_{n}$ such that $f(x_{i})=x_{\theta(i)}$. It means,
\begin{equation*}
Perm(f)=\{\theta:f(x_{i})=x_{\theta(i)}, x_{i},x_{\theta(i)} \in P_{n}\}
\end{equation*}
    \end{definition}

In order to make a relationship between permutations and Sharkovskii's Theorem, it is necessary to define a relationship between permutations of different orders.

    \begin{definition}
Let $\theta$ and $\eta$ be permutations. Say $\theta$ dominates $\eta$, denoted by $\theta \lhd \eta$, if $\{f:\theta\in Perm(f)\}$ is contained in $\{f:\eta\in Perm(f)\}$
    \end{definition}

	\begin{definition}[Simple Permutations]
A permutation is considered simple (see \cite{Block1}, \cite{Block2}) if it satisfies one of the following conditions according to its order\medskip

	\begin{enumerate}

	\item
	\emph{(Odd Order)}
These permutations are also known as Stefan Orbits (see \cite{Stefan}). Consider $k=2n-1$. $\theta$ is a simple permutation if $\theta\in C_{2k-1}=\{\alpha_{2n-1},\beta_{2n-1}\}$, where
\begin{eqnarray*}
\alpha_{2n-1} & = & (1,2n-1,n,n-1,n+1,n-2,n+2,\dots,2,2n-2)\\
 & = & \left(\begin{array}{ccccccccc}
1 & 2 & \dots & n-1 & n & n+1 & \dots & 2n-2 & 2n-1\\
2n-1 & 2n-2 & \dots & n+1 & n-1 & n-2 & \dots & 1 & n\end{array}\right)\
\end{eqnarray*}
or
\begin{eqnarray*}
\beta_{2n-1} & = & (1,n,n+1,n-1,n+2,n-2,\dots,2,2n-1)\\
 & = & \left(\begin{array}{ccccccccc}
1 & 2 & \dots & n-1 & n & n+1 & \dots & 2n-2 & 2n-1\\
n & 2n-1 & \dots & n+2 & n+1 & n-1 & \dots & 2 & 1\end{array}\right)\
\end{eqnarray*}\medskip

	\item
	\emph{(Order a Power of Two)}
A permutation of order a power of two is simple (belongs to $C_{2^{r}}$) if it satisfies
\begin{enumerate}
\item
$\theta^{2^{j}}[P(2^{n},2^{j},i)]=P(2^{n},2^{j},i)$
\item
$\theta^{2^{j}}[P(2^{n},2^{k+1},j)]$ has empty intersection with $P(2^{n},2^{k+1},j)$
\end{enumerate}\medskip

	\item\medskip

	\emph{(Mixed Order)}
From this part and on, a permutation of mixed order will be understood as
 a permutation of order $r2^{m}$, where $r$  is an odd integer. The permutations of mixed order are defined by mixing some properties of the previous permutations. A permutation of mixed order $r2^{m}$ is simple (Belongs to Perm(f)) if it satisfies the following two conditions:

\begin{enumerate}
\item $\theta\left[P\left(r2^{m},2^{m},j\right)\right]=P\left(r2^{m},2^{m},\sigma\left(j\right)\right)$, where
$\sigma$  is a simple element of $C_{2^{r}}$;\medskip

\item $\theta^{2^{m}}$ restricted to $P\left(r2^{m},2^{m},j\right)$  is simple for every $j$.
\end{enumerate}

	\end{enumerate}
	\end{definition}

The following examples show some permutations of every case

\begin{example} The permutations
$$\alpha_{5}=\left(\begin{array}{ccccc}
1 & 2 & 3 & 4 & 5\\
5 & 4 & 2 & 1 & 3\end{array}\right),\quad
\beta_{5}=\left(\begin{array}{ccccc}
1 & 2 & 3 & 4 & 5\\
3 & 5 & 4 & 2 & 1\end{array}\right)$$ are the Stefan Orbits of order $5$.
\end{example}

\begin{example}The permutation
$$\theta=\left(\begin{array}{cccc}
1 & 2 & 3 & 4\\
3 & 4 & 2 & 1 \end{array}\right)$$ is a simple permutation of order 4 because $\theta\{1,2\}=\{3,4\}$, $\theta\{3,4\}=\{1,2\}$, $\theta\{1,2,3,4\}=\{1,2,3,4\}$, $\theta^{2}\{1,2\}=\{1,2\}$, $\theta^{2}\{3,4\}=\{3,4\}$, $\theta^{2}\{1\}=\{2\}$,  $\theta^{2}\{2\}=\{1\}$,  $\theta^{2}\{3\}=\{4\}$,  $\theta^{2}\{4\}=\{3\}.$
\end{example}

\begin{example}
Let $\eta=\left(\begin{array}{cccccc}
1 & 2 & 3 & 4 & 5 & 6\\
6 & 5 & 4 & 1 & 3 & 2\end{array}\right)$.

Considering
$\alpha=\left(\begin{array}{ccc}
1 & 2 & 3\\
3 & 1 & 2\end{array}\right)$
and
$\beta=\left(\begin{array}{ccc}
1 & 2 & 3\\
2 & 3 & 1\end{array}\right)$
simple permutations of odd order and
 $\phi=\left(\begin{array}{cc}
1 & 2\\
2 & 1\end{array}\right)$, it is easy to check that:
 \begin{eqnarray*}
\theta\left[P\left(6,2,1\right)\right] & = & \theta\left[\left\{ 1,2,3\right\} \right]=\left\{ 6,5,4\right\} =\left\{ 4,5,6\right\} =P\left(6,2,\phi\left(1\right)\right)=P\left(6,2,2\right)\end{eqnarray*}
\begin{eqnarray*}
\theta\left[P\left(6,2,2\right)\right] & = & \theta\left[\left\{ 4,5,6\right\} \right]=\left\{ 1,3,2\right\} =\left\{ 1,2,3\right\} =P\left(6,2,\phi\left(2\right)\right)=P\left(6,2,1\right)\end{eqnarray*}
and
$\theta^{2}=\left(\begin{array}{cccccc}
1 & 2 & 3 & 4 & 5 & 6\\
2 & 3 & 1 & 6 & 4 & 5\end{array}\right)$. Such as we will see, this permutation can be written by using right Pasting as two disjoint
 cycles of simple permutations,
 $\left(\begin{array}{ccc}
1 & 2 & 3\\
2 & 3 & 1\end{array}\right)$
 and
 $\left(\begin{array}{ccc}
1 & 2 & 3\\
3 & 1 & 2\end{array}\right)$.

\end{example}

	\subsection{Pasting and Reversing}

Pasting and Reversing operations have been defined and applied on integer numbers (\cite{Pegamiento}), rings over conmutative fields (\cite{preprint}), cycles, and permutations (\cite{Genealogia}).  These operations will be used to describe the genealogy of the permutations and a procedure to generate simple permutations with mixed order.

	\begin{definition}[Pasting and Reversing Cycles]
Let be $\theta \in S_{n}$ such that $$\theta = (u) \diamond (v) = (i_{1}, \dots, i_{k})(i_{k+1}, \dots, i_{n}),$$ where $1 \leq k \leq n, i_{j} \in {1, \dots, n}$, $(u)$ and $(v)$ are disjoint cycles. The Pasting of $(u)$ with $(v)$ is a $n$-cycle defined as follows:

\begin{equation*}
(u)\diamond(v)=(u,v)=(i_{1},\dots,i_{k},i_{k+1},\dots,i_{n})
\end{equation*}

The Reversing of $(u)=(i_{1},\dots,i_{k})$ is defined as follows:

\begin{equation*}
\widetilde{(u)}=(i_{k},i_{k-1},\dots,i_{2},i_{i}).
\end{equation*}
	\end{definition}

	\begin{definition}[Pasting and Reversing Permutations]
Let be $\alpha\in S_{m}$, $\beta\in S_{n}$. The Pasting of $\alpha$ and $\beta$ are permutations in $S_{n+m}$ defined as follows: \medskip

\begin{itemize}
\item
Left Pasting of $\alpha$ and $\beta$ ($\alpha\mid\diamond\beta$)
\begin{eqnarray*}
\alpha\mid\diamond\beta & = & \left(\begin{array}{cccccc}
1 & \dots & m & m+1 & \dots & m+n\\
\alpha (1)+n & \dots & \alpha (m)+n & \beta (1) & \dots & \beta (n)\end{array}\right)\
\end{eqnarray*}\medskip

\item
Right Pasting of $\alpha$ and $\beta$ ($\alpha\diamond\mid\beta$)
\begin{eqnarray*}
\alpha\diamond\mid\beta & = & \left(\begin{array}{cccccc}
1 & \dots & m & m+1 & \dots & m+n\\
\alpha (1) & \dots & \alpha (m) & \beta (1)+m & \dots & \beta (n)+m\end{array}\right).\
\end{eqnarray*}
\end{itemize}

Reversing of $\alpha$ ($\widetilde{\alpha}$) is a permutation of $S_{m}$ defined as
\begin{eqnarray*}
\widetilde{\alpha} & = & \left(\begin{array}{ccccc}
1 & 2 & \dots & m-1 & m\\
\alpha (m) & \alpha(m)-1 & \dots & \alpha (2) & \alpha (1)\end{array}\right).\
\end{eqnarray*}
	\end{definition}

The following theorems state the properties (algebraic structure) of these new operations, see also \cite{Genealogia}.

	\begin{theorem}[Associative Property of Pasting]
Let be $\alpha\in S_{m}, \beta\in S_{n}, \gamma\in S_{k}, (u)=(i_{1},\dots,i_{s_{1}}), (v)=(i_{1},\dots,i_{s_{2}}), (w)=(i_{1},\dots,i_{s_{3}})$, then \medskip

\begin{enumerate}
\item
$(\alpha\mid\diamond\beta)\mid\diamond\gamma=\alpha\mid\diamond(\beta\mid\diamond\gamma)$ \medskip

\item
$(\alpha\diamond\mid\beta)\diamond\mid\gamma=\alpha\diamond\mid(\beta\diamond\mid\gamma)$
\medskip

\item
$((u)\diamond(v))\diamond(w)=(u)\diamond((v)\diamond(w))$
\end{enumerate}
	\end{theorem}

	\begin{proof}
\begin{enumerate}
\item
Consider $\alpha\in S_{m}, \beta\in S_{n}, \gamma\in S_{k}$
\begin{eqnarray*}
\alpha\mid\diamond\beta & = & \left(\begin{array}{cccccc}
1 & \dots & m & m+1 & \dots & m+n\\
\alpha (1)+n & \dots & \alpha(m)+n & \beta (1) & \dots & \beta (n)\end{array}\right)\\
(\alpha\mid\diamond\beta)\mid\diamond\gamma & = & \Bigg(\begin{array}{cccccc}
1 & \dots & m & m+1 & \dots & m+n\\
\alpha (1)+n+k & \dots & \alpha(m)+n+k & \beta (1)+k & \dots & \beta (n)+k \end{array}\\&&\\
&  & \begin{array}{ccc}
m+n+1 & \dots & m+n+k\\
\gamma(1) & \dots & \gamma (k) \end{array}\Bigg)\\
\end{eqnarray*}
On the other hand
\begin{eqnarray*}
\beta\mid\diamond\gamma & = & \left(\begin{array}{cccccc}
1 & \dots & n & n+1 & \dots & n+k\\
\beta (1)+k & \dots & \beta(n)+k & \gamma (1) & \dots & \gamma (k)\end{array}\right)\\
\alpha\mid\diamond(\beta\mid\diamond\gamma) & = & \Bigg(\begin{array}{cccccc}
1 & \dots & m & m+1 & \dots & m+n\\
\alpha (1)+n+k & \dots & \alpha(m)+n+k & \beta (1)+k & \dots & \beta (n)+k \end{array}\\&&\\
&  & \begin{array}{ccc}
m+n+1 & \dots & m+n+k\\
\gamma(1) & \dots & \gamma (k) \end{array}\Bigg)\\
\end{eqnarray*}
Therefore $(\alpha\mid\diamond\beta)\mid\diamond\gamma=\alpha\mid\diamond(\beta\mid\diamond\gamma)$.
\medskip

\item
The proof of this item is analogous to (1.).
\medskip

\item
Consider $(u)=(i_{1},\dots,i_{s_{1}}), (v)=(i_{1},\dots,i_{s_{2}}), (w)=(i_{1},\dots,i_{s_{3}})$
\begin{eqnarray*}
(u)\diamond(v) & = & (i_{1},\dots,i_{s_{1}},i_{s_{1}+1},\dots,i_{s_{1}+s_{2}})\\
((u)\diamond(v))\diamond(w) & = & (i_{1},\dots,i_{s_{1}},i_{s_{1}+1},\dots,i_{s_{1}+s_{2}},i_{s_{1}+s_{2}+1},\dots,i_{s_{1}+s_{2}+s_{3}})\\
\end{eqnarray*}
On the other hand
\begin{eqnarray*}
(v)\diamond(w) & = & (i_{1},\dots,i_{s_{2}},i_{s_{2}+1},\dots,i_{s_{2}+s_{3}})\\
(u)\diamond((v)\diamond(w)) & = & (i_{1},\dots,i_{s_{1}},i_{s_{1}+1},\dots,i_{s_{1}+s_{2}},i_{s_{1}+s_{2}+1},\dots,i_{s_{1}+s_{2}+s_{3}})\\
\end{eqnarray*}
\end{enumerate}
	\end{proof}

	\begin{theorem}[Double Reversing]
Let be $\alpha\in S_{m}, (u)=(i_{1},\dots,i_{s_{1}})$, then
\begin{enumerate}
\item
$\widetilde{\widetilde{\alpha}}=\alpha$
\item
$\widetilde{\widetilde{(u)}}=(u)$
\end{enumerate}
	\end{theorem}

	\begin{proof}

	\begin{enumerate}
	\item
Consider $\alpha\in S_{m}$
\begin{eqnarray*}
\alpha & =  & \left(\begin{array}{ccccc}
1 & 2 & \dots & m-1 & m\\
\alpha (1) & \alpha (2) & \dots & \alpha (m-1) & \alpha (m)\end{array}\right)\\
\widetilde{\alpha} & = & \left(\begin{array}{ccccc}
1 & 2 & \dots & m-1 & m\\
\alpha (m) & \alpha (m-1) & \dots & \alpha (2) & \alpha (1)\end{array}\right)\\
\widetilde{\widetilde{\alpha}} & =  & \left(\begin{array}{ccccc}
1 & 2 & \dots & m-1 & m\\
\alpha (1) & \alpha (2) & \dots & \alpha (m-1) & \alpha (m)\end{array}\right)\\
\widetilde{\widetilde{\alpha}} & =  & \alpha\
\end{eqnarray*}\medskip

\item
%Consider $(u)=(i_{1},\dots,i_{s_{1}})$
\begin{eqnarray*}
(u) & = & (i_{1},i_{2},\dots,i_{s_{1}})\\
\widetilde{(u)} & = & (i_{s_{1}},\dots,i_{2},i_{1})\\
\widetilde{\widetilde{(u)}} & = & (i_{1},i_{2},\dots,i_{s_{1}})\\
\widetilde{\widetilde{(u)}} & = & (u)\\\end{eqnarray*}

	\end{enumerate}
	\end{proof}

To finish this section we present the following result.

	\begin{theorem}[Reversing of Pasting Property]
Let be $\alpha\in S_{m}, \beta\in S_{n}$, then \medskip

\begin{enumerate}
\item
$\widetilde{\alpha\mid\diamond\beta}=\widetilde{\beta}\diamond\mid\widetilde{\alpha}$ \medskip

\item$\widetilde{\alpha\diamond\mid\beta}=\widetilde{\beta}\mid\diamond\widetilde{\alpha}$
\end{enumerate}
	\end{theorem}
\begin{proof}

Consider $\alpha\in S_{m}, \beta\in S_{n},$ and $m < n$.

\begin{eqnarray*}
\widetilde{\alpha\mid\diamond\beta} & = & \left(\begin{array}{cccccc}
1 & \dots & m & m+1 & \dots & m+n\\
\alpha (1)+n & \dots & \alpha (m)+n & \beta (1) & \dots & \beta (n)\end{array}\right)\\
 & = & \left(\begin{array}{cccccccc}
1 & \dots & n & n+1 & \dots & m & \dots & m+n\\
\beta (n) & \dots & \beta (1) & \alpha (m)+n & \dots & \alpha (1) & \dots & \alpha (1)+n\end{array}\right)\\
 & = & \left(\begin{array}{ccc}
1 & \dots & n\\
\beta (n) & \dots & \beta (1)\end{array}\right)\diamond\mid\left(\begin{array}{ccc}
1 & \dots & m\\
\alpha (m) & \dots & \alpha (1)\end{array}\right)\\
 & = & \widetilde{\left(\begin{array}{ccc}
1 & \dots & n\\
\beta (1) & \dots & \beta (n)\end{array}\right)}\diamond\mid\widetilde{\left(\begin{array}{ccc}
1 & \dots & m\\
\alpha (1) & \dots & \alpha (m)\end{array}\right)}\\
 & = & \widetilde{\beta}\diamond\mid\widetilde{\alpha}\\
\end{eqnarray*}\medskip

\begin{eqnarray*}
\widetilde{\alpha\diamond\mid\beta} & = & \left(\begin{array}{cccccc}
1 & \dots & m & m+1 & \dots & m+n\\
\alpha (1) & \dots & \alpha (m) & \beta (1)+m & \dots & \beta (n)+m\end{array}\right)\\
 & = & \left(\begin{array}{cccccc}
1 & \dots & n & n+1 & \dots & m+n\\
\beta (n)+m & \dots & \beta (1)+m & \alpha (m) & \dots & \alpha (1)\end{array}\right)\\
 & = & \left(\begin{array}{ccc}
1 & \dots & n\\
\beta (n) & \dots & \beta (1)\end{array}\right)\mid\diamond\left(\begin{array}{ccc}
1 & \dots & m\\
\alpha (m) & \dots & \alpha (1)\end{array}\right)\\
 & = & \widetilde{\left(\begin{array}{ccc}
1 & \dots & n\\
\beta (1) & \dots & \beta (n)\end{array}\right)}\mid\diamond\widetilde{\left(\begin{array}{ccc}
1 & \dots & m\\
\alpha (1) & \dots & \alpha (m)\end{array}\right)}\\
 & = & \widetilde{\beta}\mid\diamond\widetilde{\alpha}\\
\end{eqnarray*}\medskip

\end{proof}

	\section{Genealogy of Simple Permutations}

The definition of Block's orbits is also the genealogy for the odd case, given by $\alpha_{n}$ and $\beta_{n}$. There exists two genealogical lines, one per each kind of simple permutation. The other two cases (A power of two, Mixed order) will be developed subsequently.

	\subsection{A power of Two}
The genealogy of simple permutations with order a power of two has been described by Bernhardt \cite{Bernhardt} through transpositions and Acosta-Hum\'anez \cite{Genealogia} through Pasting and Reversing.

		\subsubsection{Transpositions}
Some definitions are required to construct this genealogy.
	\begin{definition}
If $\theta$ belongs to $Sim(2^{n})$ then $\theta^{*}$, an element of $S_{2^{n+1}}$, is defined by
\begin{eqnarray*}
\theta^{*}(2k)=2\theta{k}, & & \theta^{*}(2k-1)=2\theta(k)-1.\\
\end{eqnarray*}
The permutation $\theta^{*}$ consists of two $2^{n}-$cycles, moreover, $\theta^{*}\triangleleft\theta$
	\end{definition}

	\begin{definition}
If $\theta$ belongs to $Sim(2^{n})$ then $\theta_{*}$, an element of $S_{2^{n-1}}$, is defined by
\begin{eqnarray*}
\theta_{*}(k)=\left\lfloor\frac{1}{2}(\theta{(2k)}+1)\right\rfloor\\
\end{eqnarray*}
where $\lfloor x \rfloor$ is the greatest integer less than or equal to $x$.
	\end{definition}	

	\begin{definition}
Let $\rho_{s}$ denote the transposition
$\left(\begin{array}{cc}
2s-1 & 2s\\
2s & 2s-1\end{array}\right).$
	\end{definition}

The following theorem states the existence of predecessors and successors of a given permutation and a way to obtain them.

	\begin{theorem}
Let be $\eta \in$ Sim($2^{n+1}$),$\theta\in$Sim($2^{n}$),$\phi\in$Sim($2^{n-1}$).
If $\eta \triangleleft \theta \triangleleft \phi$, then $\eta=\theta^{*}\circ\rho_{i_{1}}\circ\dots\circ\rho_{i_{2m-1}}$ and $\rho=\theta_{*}$.
	\end{theorem}

		\subsubsection{Pasting and Reversing}
The following theorem, due to the first author in \cite{Genealogia}, shows a way to construct the same genealogy by using Pasting and Reversing. It is an important result because it is the first use of Pasting and Reversing in combinatorial dynamics.

	\begin{theorem}
Let be $n=2^{k+1}$, $k\in\mathbb{Z}$, $\theta_{1}=\phi_{1}=(1)$, $\theta_{n}$ and $\phi_{n}$ permutations as follows:
\begin{eqnarray*}
\theta_{n} & = & e_{\frac{n}{2}\mid\diamond\theta_{\frac{n}{2}}}\\
\phi_{n} & = & \widetilde{e_{\frac{n}{2}}}\mid\diamond\widetilde{\phi_{\frac{n}{2}}}\\
 & = & \widetilde{\phi_{\frac{n}{2}}\diamond\mid e_{\frac{n}{2}}},\\
\end{eqnarray*}
where
\begin{eqnarray*}
e_{\frac{n}{2}} & = & \left(\begin{array}{cccccc}
1 & 2 & 3 & \dots & \frac{n-2}{2} & \frac{n}{2}\\
1 & 2 & 3 & \dots & \frac{n-2}{2} & \frac{n}{2}\\
\end{array}\right).\\
\end{eqnarray*}

\begin{enumerate}
\item
$\varphi=(\theta_{n}\circ\rho_{k})\circ(\phi_{n}\circ\rho_{j}),\forall\varphi\in Sim(n)$, where $\rho_{k}$ and $\rho_{j}$ are compositions of odd length transpositions,
\item
$\theta_{n}\cong\phi_{n}$, $\theta_{2n}\triangleleft\theta_{n}\triangleleft\theta_{\frac{n}{2}},$ $\phi_{2n}\triangleleft\phi_{n}\triangleleft\phi_{\frac{n}{2}},$
\item
$\theta_{n},\phi_{n}\in Sim(n),$
\item
$(\theta_{2n})_{*}=\theta_{n}=(\theta_{\frac{n}{2}})^{*}\circ\rho_{\frac{n}{2}}$, $(\phi_{2n})_{*}=\phi_{n}=(\phi_{\frac{n}{2}})^{*}\circ\rho_{1}\dots\rho_{\frac{n-2}{2}}.$
\end{enumerate}
	\end{theorem}

	\subsection{Mixed Order: Case $4n+2$}

Pasting operation and simple permutations of odd order will be used to construct
 mixed order permutations. Part (ii) of Definition 1 allows to suggest a way to construct simple permutations $\theta$ taking as reference $\theta^{2}$ and its relationship with $\alpha,\beta$.

This paper will introduce that procedure and state a first theorem that describes a first genealogy relationship between simple permutations of order $2r$.  In this particular case, as
 $\phi=\left(\begin{array}{cc}
1 & 2\\
2 & 1\end{array}\right)$ for all of the cases, the condition (ii) of definition 1.6 is equivalent
 to represent $\theta^{2}$ as two simple permutations of order $r$ pasted by right. If it is true that for any permutation $\theta$ there exists only one permutation $\theta^{2}$, it does not happen in the other way, it means, given $\theta^{2}$ it is possible to find $\theta_{1}$, $\theta_{2}$ such that $\theta_{1}^{2}$=$\theta_{2}^{2}$=$\theta^{2}$. So, there exists more than 4 simple permutations with mixed order ${4n+2}$. The cardinal of this set will be stated ahead.
 In order to explain the procedure, the particular cases 6 and 10 will be shown.

	\subsubsection{Order 6}
There are 12 simple permutations of order 6 in 4 subsets, each one of these subsets correspond
 to $\theta^{2}$ by right Pasting of $\alpha_{3}$ and $\beta_{3}$ as follows:

\begin{enumerate}
\item
\begin{eqnarray*}
\theta_{\alpha\alpha6} & = & \left(\begin{array}{cccccc}
1 & 2 & 3 & 4 & 5 & 6\\
6 & 4 & 5 & 1 & 2 & 3\end{array}\right)\\
\theta_{\alpha\alpha5} & = & \left(\begin{array}{cccccc}
1 & 2 & 3 & 4 & 5 & 6\\
5 & 6 & 4 & 2 & 3 & 1\end{array}\right)\\
\theta_{\alpha\alpha4} & = & \left(\begin{array}{cccccc}
1 & 2 & 3 & 4 & 5 & 6\\
4 & 5 & 6 & 3 & 1 & 2\end{array}\right)\end{eqnarray*}
considering
$\alpha=\left(\begin{array}{ccc}
1 & 2 & 3\\
3 & 1 & 2\end{array}\right)$ as simple permutation of odd order ${3}$. For ${i=4,5,6}$
\begin{eqnarray*}
\theta_{\alpha\alpha i}\left[P\left(6,2,1\right)\right] & = & \theta_{\alpha\alpha i}\left[\left\{ 1,2,3\right\} \right]\\
& = & \left\{ 4,5,6\right\}\\
& = & P\left(6,2,\phi\left(1\right)\right)\\
& = & P\left(6,2,2\right)\end{eqnarray*}
\begin{eqnarray*}
\theta_{\alpha\alpha i}\left[P\left(6,2,2\right)\right] & = & \theta_{\alpha\alpha i}\left[\left\{ 4,5,6\right\} \right]\\
& = & \left\{ 1,2,3\right\}\\
& = & P\left(6,2,\phi\left(2\right)\right)\\
& = & P\left(6,2,1\right)\end{eqnarray*}
and
$\theta_{\alpha\alpha}^{2}=\left(\begin{array}{cccccc}
1 & 2 & 3 & 4 & 5 & 6\\
3 & 1 & 2 & 6 & 4 & 5\end{array}\right)$. This permutation can be written by using right Pasting of $\alpha=\left(\begin{array}{ccc}
1 & 2 & 3\\
3 & 1 & 2\end{array}\right)$ and itself, it means, $\theta_{\alpha\alpha}^{2}=\alpha\diamond\mid\alpha$.

\item
\begin{eqnarray*}
\theta_{\beta\beta 4} & = & \left(\begin{array}{cccccc}
1 & 2 & 3 & 4 & 5 & 6\\
4 & 5 & 6 & 2 & 3 & 1\end{array}\right)\\
\theta_{\beta\beta 5} & = & \left(\begin{array}{cccccc}
1 & 2 & 3 & 4 & 5 & 6\\
5 & 6 & 4 & 1 & 2 & 3\end{array}\right)\\
\theta_{\beta\beta 6} & = & \left(\begin{array}{cccccc}
1 & 2 & 3 & 4 & 5 & 6\\
6 & 4 & 5 & 3 & 1 & 2\end{array}\right)
\end{eqnarray*}
considering
$\beta=\left(\begin{array}{ccc}
1 & 2 & 3\\
2 & 3 & 1\end{array}\right)$ as simple permutation of odd order ${3}$. For ${i=4,5,6}$
\begin{eqnarray*}
\theta_{\beta\beta 1}\left[P\left(6,2,1\right)\right] & = & \theta_{\beta\beta 1}\left[\left\{ 1,2,3\right\} \right]\\
& = & \left\{ 4,5,6\right\}\\
& = & P\left(6,2\phi\left(1\right)\right)\\
& = & P\left(6,2,2\right)\end{eqnarray*}

\begin{eqnarray*}
\theta_{\beta\beta 2}\left[P\left(6,2,2\right)\right] & = & \theta_{\beta\beta 2}\left[\left\{ 4,5,6\right\} \right]\\
& = & \left\{ 1,2,3\right\}\\
& = & P\left(6,2,\phi\left(2\right)\right)\\
& = & P\left(6,2,1\right)\end{eqnarray*}
and
$\theta_{\beta\beta}^{2}=\left(\begin{array}{cccccc}
1 & 2 & 3 & 4 & 5 & 6\\
2 & 3 & 1 & 5 & 6 & 4\end{array}\right)$. This permutation can be written by using right Pasting of
$\beta=\left(\begin{array}{ccc}
1 & 2 & 3\\
2 & 3 & 1\end{array}\right)$ and itself, it means,
$\theta_{\beta\beta}^{2}=\beta\diamond\mid\beta$.

\item
\begin{eqnarray*}
\theta_{\alpha\beta 4} & = & \left(\begin{array}{cccccc}
1 & 2 & 3 & 4 & 5 & 6\\
4 & 6 & 5 & 3 & 1 & 2\end{array}\right)\\
\theta_{\alpha\beta 5} & = & \left(\begin{array}{cccccc}
1 & 2 & 3 & 4 & 5 & 6\\
5 & 4 & 6 & 1 & 3 & 2\end{array}\right)\\
\theta_{\alpha\beta 6} & = & \left(\begin{array}{cccccc}
1 & 2 & 3 & 4 & 5 & 6\\
6 & 5 & 4 & 2 & 1 & 3\end{array}\right)\end{eqnarray*}
considering
$\alpha=\left(\begin{array}{ccc}
1 & 2 & 3\\
3 & 1 & 2\end{array}\right)$ and
$\beta=\left(\begin{array}{ccc}
1 & 2 & 3\\
2 & 3 & 1\end{array}\right)$ as simple permutations of odd order ${3}$. For ${i=4,5,6}$

\begin{eqnarray*}
\theta_{\alpha\beta i}\left[P\left(6,2,1\right)\right] & = & \theta_{\alpha\beta i}\left[\left\{ 1,2,3\right\} \right]\\
& = & \left\{ 4,5,6\right\}\\
& = & P\left(6,2,\phi\left(1\right)\right)\\
& = & P\left(6,2,2\right)\end{eqnarray*}

\begin{eqnarray*}
\theta_{\alpha\beta i}\left[P\left(6,2,2\right)\right] & = & \theta_{\alpha\beta i}\left[\left\{ 4,5,6\right\} \right]\\
& = & \left\{ 1,2,3\right\}\\
& = & P\left(6,2,\phi\left(2\right)\right)\\
& = & P\left(6,2,1\right)\end{eqnarray*}

and
$\theta_{\alpha\beta}^{2}=\left(\begin{array}{cccccc}
1 & 2 & 3 & 4 & 5 & 6\\
3 & 1 & 2 & 5 & 6 & 4\end{array}\right)$.  This permutation can be written by using right Pasting of
$\alpha=\left(\begin{array}{ccc}
1 & 2 & 3\\
3 & 1 & 2\end{array}\right)$ and
$\beta=\left(\begin{array}{ccc}
1 & 2 & 3\\
2 & 3 & 1\end{array}\right)$, it means,
$\theta^{2}=\alpha\diamond\mid\beta$.

\item
\begin{eqnarray*}
\theta_{\beta\alpha 4} & = & \left(\begin{array}{cccccc}
1 & 2 & 3 & 4 & 5 & 6\\
4 & 6 & 5 & 2 & 1 & 3\end{array}\right)\\
\theta_{\beta\alpha 5} & = & \left(\begin{array}{cccccc}
1 & 2 & 3 & 4 & 5 & 6\\
5 & 4 & 6 & 3 & 2 & 1\end{array}\right)\\
\theta_{\beta\alpha 6} & = & \left(\begin{array}{cccccc}
1 & 2 & 3 & 4 & 5 & 6\\
6 & 5 & 4 & 1 & 3 & 2\end{array}\right)\end{eqnarray*}
considering
$\alpha=\left(\begin{array}{ccc}
1 & 2 & 3\\
3 & 1 & 2\end{array}\right)$ and
$\beta=\left(\begin{array}{ccc}
1 & 2 & 3\\
2 & 3 & 1\end{array}\right)$ as simple permutations of odd order ${3}$. For ${i=4,5,6}$

\begin{eqnarray*}
\theta_{\beta\alpha i}\left[P\left(6,2,1\right)\right] & = & \theta_{\alpha\beta}\left[\left\{ 1,2,3\right\} \right]\\
& = & \left\{ 4,5,6\right\}\\
& = & P\left(6,2,\phi\left(1\right)\right)\\
& = & P\left(6,2,2\right)\end{eqnarray*}

\begin{eqnarray*}
\theta_{\beta\alpha i}\left[P\left(6,2,2\right)\right] & = & \theta_{\beta\alpha}\left[\left\{ 4,5,6\right\} \right]\\
& = & \left\{ 1,2,3\right\}\\
& = & P\left(6,2,\phi\left(2\right)\right)\\
& = & P\left(6,2,1\right)\end{eqnarray*}

and
$\theta_{\beta\alpha}^{2}=\left(\begin{array}{cccccc}
1 & 2 & 3 & 4 & 5 & 6\\
2 & 3 & 1 & 6 & 4 & 5\end{array}\right)$.  This permutation can be written by using right Pasting of
$\beta=\left(\begin{array}{ccc}
1 & 2 & 3\\
2 & 3 & 1\end{array}\right)$ and
$\alpha=\left(\begin{array}{ccc}
1 & 2 & 3\\
3 & 1 & 2\end{array}\right)$, it means,
$\theta^{2}=\beta\diamond\mid\alpha$.

\end{enumerate}

For all of these simple permutations $\theta$, $\theta$(1) will be 4, 5 or 6. It means, $\theta$(1) takes values between ${2n+2}$ and ${4n+2}$.

\subsubsection{Order 10}
Following the previous procedure, It is possible to construct 20 simple permutations of order 10,  also arranged in 4 subsets, each one of them corresponds to $\theta^{2}$ by right Pasting of $\alpha_{5}$ and $\beta_{5}$ as follows:

\begin{enumerate}

\item
\begin{eqnarray*}
\theta_{\alpha\alpha 6} & = & \left(\begin{array}{cccccccccc}
1 & 2 & 3 & 4 & 5 & 6 & 7 & 8 & 9 & 10\\
6 & 7 & 8 & 9 & 10 & 5 & 4 & 2 & 1 & 3\end{array}\right)\\
\theta_{\alpha\alpha 7} & = & \left(\begin{array}{cccccccccc}
1 & 2 & 3 & 4 & 5 & 6 & 7 & 8 & 9 & 10\\
7 & 10 & 6 & 8 & 9 & 2 & 5 & 1 & 3 & 4\end{array}\right)\\
\theta_{\alpha\alpha 8} & = & \left(\begin{array}{cccccccccc}
1 & 2 & 3 & 4 & 5 & 6 & 7 & 8 & 9 & 10\\
8 & 6 & 9 & 10 & 7 & 4 & 3 & 5 & 2 & 1\end{array}\right)\\
\theta_{\alpha\alpha 9} & = & \left(\begin{array}{cccccccccc}
1 & 2 & 3 & 4 & 5 & 6 & 7 & 8 & 9 & 10\\
9 & 8 & 10 & 7 & 6 & 3 & 1 & 4 & 5 & 2\end{array}\right)\\
\theta_{\alpha\alpha 10} & = & \left(\begin{array}{cccccccccc}
1 & 2 & 3 & 4 & 5 & 6 & 7 & 8 & 9 & 10\\
10 & 9 & 7 & 6 & 8 & 1 & 2 & 3 & 4 & 5\end{array}\right)\end{eqnarray*}
considering
$\alpha=\left(\begin{array}{ccccc}
1 & 2 & 3 & 4 & 5\\
5 & 4 & 2 & 1 & 3\end{array}\right)$ as simple permutation of odd order ${5}$. For ${i=6,7,8,9,10}$
\begin{eqnarray*}
\theta_{\alpha\alpha i}\left[P\left(10,2,1\right)\right] & = & \theta_{\alpha\alpha i}\left[\left\{ 1,2,3,4,5\right\} \right]\\
& = & \left\{ 6,7,8,9,10\right\}\\
& = & P\left(10,2,\phi\left(1\right)\right)\\
& = & P\left(10,2,2\right)\end{eqnarray*}
\begin{eqnarray*}
\theta_{\alpha\alpha i}\left[P\left(10,2,2\right)\right] & = & \theta_{\alpha\alpha i}\left[\left\{ 6,7,8,9,10\right\} \right]\\
& = & \left\{ 1,2,3,4,5\right\}\\
& = & P\left(10,2\phi\left(2\right)\right)\\
& = & P\left(10,2,1\right)\end{eqnarray*}
and
$\theta_{\alpha\alpha}^{2}=\left(\begin{array}{cccccccccc}
1 & 2 & 3 & 4 & 5 & 6 & 7 & 8 & 9 & 10\\
5 & 4 & 2 & 1 & 3 & 10 & 9 & 7 & 6 & 8\end{array}\right)$. This permutation can be written by using right Pasting of $\alpha=\left(\begin{array}{ccccc}
1 & 2 & 3 & 4 & 5\\
5 & 4 & 2 & 1 & 3\end{array}\right)$ and itself, it means, $\theta_{\alpha\alpha}^{2} = \alpha\diamond\mid\alpha$.

\item
\begin{eqnarray*}
\theta_{\beta\beta 6} & = & \left(\begin{array}{cccccccccc}
1 & 2 & 3 & 4 & 5 & 6 & 7 & 8 & 9 & 10\\
6 & 7 & 8 & 9 & 10 & 3 & 5 & 4 & 2 & 1\end{array}\right)\\
\theta_{\beta\beta 7} & = & \left(\begin{array}{cccccccccc}
1 & 2 & 3 & 4 & 5 & 6 & 7 & 8 & 9 & 10\\
7 & 8 & 10 & 6 & 9 & 2 & 3 & 5 & 1 & 4\end{array}\right)\\
\theta_{\beta\beta 8} & = & \left(\begin{array}{cccccccccc}
1 & 2 & 3 & 4 & 5 & 6 & 7 & 8 & 9 & 10\\
8 & 10 & 9 & 7 & 6 & 1 & 2 & 3 & 4 & 5\end{array}\right)\\
\theta_{\beta\beta 9} & = & \left(\begin{array}{cccccccccc}
1 & 2 & 3 & 4 & 5 & 6 & 7 & 8 & 9 & 10\\
9 & 6 & 7 & 10 & 8 & 5 & 4 & 1 & 3 & 2\end{array}\right)\\
\theta_{\beta\beta 10} & = & \left(\begin{array}{cccccccccc}
1 & 2 & 3 & 4 & 5 & 6 & 7 & 8 & 9 & 10\\
10 & 9 & 6 & 8 & 7 & 4 & 1 & 2 & 5 & 3\end{array}\right)\end{eqnarray*}
considering
$\beta=\left(\begin{array}{ccccc}
1 & 2 & 3 & 4 & 5\\
3 & 5 & 4 & 2 & 1\end{array}\right)$ as simple permutation of odd order ${5}$. For ${i=4,5,6}$

\begin{eqnarray*}
\theta_{\beta\beta i}\left[P\left(10,2,1\right)\right] & = & \theta_{\beta\beta i}\left[\left\{ 1,2,3,4,5\right\} \right]\\
& = & \left\{ 6,7,8,9,10\right\}\\
& = & P\left(10,2,\phi\left(1\right)\right)\\
& = & P\left(10,2,2\right)\end{eqnarray*}
\begin{eqnarray*}
\theta_{\beta\beta i}\left[P\left(10,2,2\right)\right] & = & \theta_{\beta\beta}\left[\left\{ 6,7,8,9,10\right\} \right]\\
& = & \left\{ 1,2,3,4,5\right\}\\
& = & P\left(10,2,\phi\left(2\right)\right)\\
& = & P\left(10,2,1\right)\end{eqnarray*}
and
$\theta_{\beta\beta}^{2}=\left(\begin{array}{cccccccccc}
1 & 2 & 3 & 4 & 5 & 6 & 7 & 8 & 9 & 10\\
3 & 5 & 4 & 2 & 1 & 8 & 10 & 9 & 7 & 6\end{array}\right)$. This permutation can be written by using right Pasting of $\beta=\left(\begin{array}{ccccc}
1 & 2 & 3 & 4 & 5\\
3 & 5 & 4 & 2 & 1\end{array}\right)$ and itself, it means, $\theta_{\beta\beta}^{2}=\beta\diamond\mid\beta$.

\item
\begin{eqnarray*}
\theta_{\alpha\beta 6} & = & \left(\begin{array}{cccccccccc}
1 & 2 & 3 & 4 & 5 & 6 & 7 & 8 & 9 & 10\\
6 & 7 & 9 & 10 & 8 & 5 & 4 & 3 & 2 & 1\end{array}\right)\\
\theta_{\alpha\beta 7} & = & \left(\begin{array}{cccccccccc}
1 & 2 & 3 & 4 & 5 & 6 & 7 & 8 & 9 & 10\\
7 & 8 & 6 & 9 & 10 & 2 & 5 & 4 & 1 & 3\end{array}\right)\\
\theta_{\alpha\beta 8} & = & \left(\begin{array}{cccccccccc}
1 & 2 & 3 & 4 & 5 & 6 & 7 & 8 & 9 & 10\\
8 & 10 & 7 & 6 & 9 & 1 & 2 & 5 & 3 & 4\end{array}\right)\\
\theta_{\alpha\beta 9} & = & \left(\begin{array}{cccccccccc}
1 & 2 & 3 & 4 & 5 & 6 & 7 & 8 & 9 & 10\\
9 & 6 & 10 & 8 & 7 & 4 & 3 & 1 & 5 & 2\end{array}\right)\\
\theta_{\alpha\beta 10} & = & \left(\begin{array}{cccccccccc}
1 & 2 & 3 & 4 & 5 & 6 & 7 & 8 & 9 & 10\\
10 & 9 & 8 & 7 & 6 & 3 & 1 & 2 & 4 & 5\end{array}\right)\end{eqnarray*}
considering
$\alpha=\left(\begin{array}{ccccc}
1 & 2 & 3 & 4 & 5\\
5 & 4 & 2 & 1 & 3\end{array}\right)$ and
$\beta=\left(\begin{array}{ccccc}
1 & 2 & 3 & 4 & 5\\
3 & 5 & 4 & 2 & 1\end{array}\right)$ as simple permutations of odd order ${5}$. For ${i=6,7,8,9,10}$

\begin{eqnarray*}
\theta_{\alpha\beta i}\left[P\left(10,2,1\right)\right] & = & \theta_{\alpha\beta i}\left[\left\{ 1,2,3,4,5\right\} \right]\\
& = & \left\{ 6,7,8,9,10\right\}\\
& = & P\left(10,2,\phi\left(1\right)\right)\\
& = & P\left(10,2,2\right)\end{eqnarray*}
\begin{eqnarray*}
\theta_{\alpha\beta i}\left[P\left(10,2,2\right)\right] & = & \theta_{\alpha\beta}\left[\left\{ 6,7,8,9,10\right\} \right]\\
& = & \left\{1,2,3,4,5\right\}\\
& = & P\left(10,2\phi\left(2\right)\right)\\
& = & P\left(10,2,1\right)\end{eqnarray*}
and
$\theta_{\alpha\beta}^{2}=\left(\begin{array}{cccccccccc}
1 & 2 & 3 & 4 & 5 & 6 & 7 & 8 & 9 & 10\\
5 & 4 & 2 & 1 & 3 & 8 & 10 & 9 & 7 & 6\end{array}\right)$. This permutation can be written by using right Pasting of
$$\alpha=\left(\begin{array}{ccccc}
1 & 2 & 3 & 4 & 5\\
5 & 4 & 2 & 1 & 3\end{array}\right),\quad
\beta=\left(\begin{array}{ccccc}
1 & 2 & 3 & 4 & 5\\
3 & 5 & 4 & 2 & 1\end{array}\right),$$ it means, $\theta_{\alpha\beta}^{2}=\alpha\diamond\mid\beta$.

\item
\begin{eqnarray*}
\theta_{\beta\alpha 6} & = & \left(\begin{array}{cccccccccc}
1 & 2 & 3 & 4 & 5 & 6 & 7 & 8 & 9 & 10\\
6 & 7 & 10 & 8 & 9 & 3 & 5 & 2 & 1 & 4\end{array}\right)\\
\theta_{\beta\alpha 7} & = & \left(\begin{array}{cccccccccc}
1 & 2 & 3 & 4 & 5 & 6 & 7 & 8 & 9 & 10\\
7 & 10 & 9 & 6 & 8 & 2 & 3 & 1 & 4 & 5\end{array}\right)\\
\theta_{\beta\alpha 8} & = & \left(\begin{array}{cccccccccc}
1 & 2 & 3 & 4 & 5 & 6 & 7 & 8 & 9 & 10\\
8 & 6 & 7 & 9 & 10 & 5 & 4 & 3 & 2 & 1\end{array}\right)\\
\theta_{\beta\alpha 9} & = & \left(\begin{array}{cccccccccc}
1 & 2 & 3 & 4 & 5 & 6 & 7 & 8 & 9 & 10\\
9 & 8 & 6 & 10 & 7 & 4 & 1 & 5 & 3 & 2\end{array}\right)\\
\theta_{\beta\alpha 10} & = & \left(\begin{array}{cccccccccc}
1 & 2 & 3 & 4 & 5 & 6 & 7 & 8 & 9 & 10\\
10 & 9 & 8 & 7 & 6 & 1 & 2 & 4 & 5 & 3\end{array}\right)\end{eqnarray*}
considering
$\beta=\left(\begin{array}{ccccc}
1 & 2 & 3 & 4 & 5\\
3 & 5 & 4 & 2 & 1\end{array}\right)$ and
$\alpha=\left(\begin{array}{ccccc}
1 & 2 & 3 & 4 & 5\\
5 & 4 & 2 & 1 & 3\end{array}\right)$ simple permutations of odd order. For ${i=6,7,8,9,10}$
\begin{eqnarray*}
\theta_{\beta\alpha i}\left[P\left(10,2,1\right)\right] & = & \theta_{\beta\alpha i}\left[\left\{ 1,2,3,4,5\right\} \right]\\
& = & \left\{ 6,7,8,9,10\right\}\\
& = & P\left(10,2,\phi\left(1\right)\right)\\
& = & P\left(10,2,2\right)\end{eqnarray*}
\begin{eqnarray*}
\theta_{\beta\alpha i}\left[P\left(10,2,2\right)\right] & = & \theta_{\beta\alpha i}\left[\left\{ 6,7,8,9,10\right\} \right]\\
& = & \left\{ 1,2,3,4,5\right\}\\
& = & P\left(10,2\phi\left(2\right)\right)\\
& = & P\left(10,2,1\right)\end{eqnarray*}
and
$\theta_{\beta\alpha}^{2}=\left(\begin{array}{cccccccccc}
1 & 2 & 3 & 4 & 5 & 6 & 7 & 8 & 9 & 10\\
3 & 5 & 4 & 2 & 1 & 10 & 9 & 7 & 6 & 8\end{array}\right)$. This permutation can be written by using right Pasting of
$\beta=\left(\begin{array}{ccccc}
1 & 2 & 3 & 4 & 5\\
3 & 5 & 4 & 2 & 1\end{array}\right)$
and
$\alpha=\left(\begin{array}{ccccc}
1 & 2 & 3 & 4 & 5\\
5 & 4 & 2 & 1 & 3\end{array}\right)$, it means, $\theta_{\beta\alpha}^{2} = \beta\diamond\mid\alpha$.

\end{enumerate}

\section{Main results}

According to definition 1.6 and subsection 2.2, there exists four subsets of simple permutation with mixed order ${4n+2}$. Each one of these subsets corresponds to one of the four possible Pasting of $\alpha_{2n+1}$, $\beta_{2n+1}$.

As $\theta\left[P\left(r2^{m},2^{m},j\right)\right] = P\left(r2^{m},2^{m},\phi\left(j\right)\right)$, where
$\phi=\left(\begin{array}{cc}
1 & 2\\
2 & 1\end{array}\right)$, it implies that the first element of those simple permutations has to be greater than ${2n+1}$. For this reason, there will be simple permutations with mixed order such that $\theta(1)=k$, with ${k=2n+2,\ldots,4n+2}$. It can be stated as a theorem.

\begin{theorem}[Cardinality]
There exists ${8n+4}$ simple permutations with mixed order $4n+2$, such that $\theta^{2}$ can be obtained by one of the following Pasting:

\begin{enumerate}
	\item
$\alpha_{2n+1}\diamond\mid\alpha_{2n+1}$\medskip

	\item
$\beta_{2n+1}\diamond\mid\beta_{2n+1}$
\medskip

	\item
$\alpha_{2n+1}\diamond\mid\beta_{2n+1}$
\medskip

	\item
$\beta_{2n+1}\diamond\mid\alpha_{2n+1}$
\end{enumerate}

\end{theorem}

\begin{proof}
A permutation with mixed order is simple if it accomplishes

\begin{enumerate}
\item
$\theta\left[P\left(r2^{m},2^{m},1\right)\right]=P\left(r2^{m},2^{m},2\right)$
\medskip

\item
$\theta\left[P\left(r2^{m},2^{m},2\right)\right]=P\left(r2^{m},2^{m},1\right)$
\medskip

\item
$\theta^{2}$ can be represented as two simple permutations of order $r$ pasted by right.
\end{enumerate}

There exist $4$ simple permutations with order $r$, so there are $4$ possible $\theta^{2}$. On the other hand, the first element of a simple permutations in this order has to be taken from $\lbrace2n+2, 2n+3, \dots, 4n+2\rbrace$ in order to accomplish the conditions (1) and (2). There are $2n+1$ ways to obtain a simple permutation associated to each one of the 4 possible $\theta^{2}$. It means, $8n+4$ simple permutations with order $4n+2$.
\end{proof}

\begin{example}
There exist 12 simple permutations with mixed order 6, as shown in 2.2.1.
\end{example}

 This theorem gives necessary conditions to construct simple permutations based on $\theta^{2}$. However, it is possible to construct some of those permutations by using $\alpha_{2n+1}$ and $\beta_{2n+1}$ without using $\theta^{2}$ thanks to Pasting operation.

\begin{theorem}
The permutations:\medskip

\begin{enumerate}
\item$\theta_{4n+2}=\alpha_{2n+1}\mid\diamond Id_{2n+1}$,\medskip

\item$\theta_{4n+2}=\beta_{2n+1}\mid\diamond Id_{2n+1}$,\medskip

\item$\theta_{4n+2}=Id_{2n+1}\mid\diamond\alpha_{2n+1}$ and\medskip

\item$\theta_{4n+2}=Id_{2n+1}\mid\diamond\beta_{2n+1}$,
\end{enumerate}
are simple permutations with mixed order

\end{theorem}

\begin{proof}
\begin{enumerate}

\item $\alpha_{2n+1}\mid\diamond Id_{2n+1}$.
There exists $$\gamma_{\alpha} : \{\alpha_{1}+n, \alpha_{2}+n,\dots,\alpha_{n}+n\} \rightarrow \{1,2,\dots,n\}$$ such that $\gamma_{\alpha}(\alpha_{i}+n)=\alpha_{i}$, where $\gamma_{\alpha}$ is a bijective function.
Let be $\theta_{\alpha \alpha}=\alpha_{2n+1}\mid\diamond Id_{2n+1}$

$\left(\begin{array}{cccccccc}
1 & 2 & \dots & n & n+1 & n+2 & \dots & 2n\\
\mbox{\ensuremath{\alpha_{1}}+n} & \mbox{\ensuremath{\alpha_{2}}+n} & \dots & \mbox{\ensuremath{\alpha_{n}}+n} & \mbox{1} & \mbox{2} & \ldots & \mbox{n}\end{array}\right)$ and check if it accomplishes the definition.

\begin{eqnarray*}
\theta_{\alpha\alpha}\left[P\left(4n+2,2,1\right)\right] & = & \theta_{\alpha\alpha}\left[1, 2, \dots, 2n+1 \right] \\
& = & \{\ensuremath{\alpha_{1}}+n,\ensuremath{\alpha_{2}}+n,\dots,\ensuremath{\alpha_{n}}+n\}\\
& = & \left\{ n+1,n+2,\dots,2n\right\}\\
& = & P\left(10,2,\phi\left(1\right)\right)\\
& = & P\left(10,2,2\right)\end{eqnarray*}

\begin{eqnarray*}
\theta_{\alpha\alpha}\left[P\left(4n+2,2,2\right)\right] & = & \theta_{\alpha\alpha}\left[ \left\{ n+1,n+2,\dots,2n\right\}\right]\\
& = &\{ 1,2,\dots,n\}\\
& = & P\left(10,2,\phi\left(2\right)\right)\\
& = & P\left(10,2,1\right)\end{eqnarray*}

and

\begin{eqnarray*}
\theta_{\alpha\alpha}^{2} &
= & \left(\begin{array}{ccccccc}
1 & 2 & \dots & n & n+1 & \dots & 2n\\
{\gamma_{\alpha} (\alpha_{1}+n)} & {\gamma_{\alpha} (\alpha_{2}+n)} & \dots & {\gamma_{\alpha} (\alpha_{n}+n)} & {\alpha_{1}+n} & \dots & {\alpha_{n}+n}\end{array}\right) \\
& = & \left(\begin{array}{cccccccc}
1 & 2 & \dots & n & n+1 & n+2 & \dots & 2n\\
{\alpha_{1}} & {\alpha_{2}} & \dots & {\alpha_{n}} & {\alpha_{1}+n} & {\alpha_{2}+n} & \dots & {\alpha_{n}+n}\end{array}\right)\\
& = & \left(\begin{array}{cccc}
1 & 2 & \dots & n\\
{\alpha_{1}} & {\alpha_{2}} & \dots & {\alpha_{n}}\end{array}\right)
\diamond\mid
\left(\begin{array}{cccccccc}
1 & 2 & \dots & n\\
{\alpha_{1}} & {\alpha_{2}} & \dots & {\alpha_{n}}\end{array}\right)\\
& = & \alpha_{2n+1}\diamond\mid\alpha_{2n+1}\end{eqnarray*}.\medskip

\item $\beta_{2n+1}\mid\diamond Id_{2n+1}$.
In order to prove the second part, it is necessary to consider $$\gamma_{\beta} : \{\beta_{1}+n, \beta_{2}+n,\dots,\beta_{n}+n\} \rightarrow \{1,2,\dots,n\}$$ such that $\gamma_{\beta}(\beta_{i}+n)=\beta_{i}$, where $\gamma_{\beta}$ that is also a bijective function. The proof is analogous to {(1)}.\medskip

\item$\theta_{4n+2} = Id_{2n+1}\mid\diamond\alpha_{2n+1}$.
Let be $$\theta_{\alpha \alpha} = Id_{2n+1} \mid\diamond\alpha_{2n+1}=\left(\begin{array}{cccccccc}
1 & 2 & \dots & n & n+1 & n+2 & \dots & 2n\\
n+1 & n+2 & \dots & 2n & \alpha_{1} & \alpha_{2} & \dots & \alpha_{1} \end{array}\right)$$ and check if it accomplishes the definition.

\begin{eqnarray*}
\theta_{\alpha\alpha}\left[P\left(4n+2,2,1\right)\right] & = & \theta_{\alpha\alpha}\left[1, 2, \dots, 2n+1 \right] \\
& = & \left\{ n+1,n+2,\dots,2n\right\}\\
& = & P\left(4n+2,2,\phi\left(1\right)\right)\\
& = & P\left(4n+2,2,2\right)\end{eqnarray*}

\begin{eqnarray*}
\theta_{\alpha\alpha}\left[P\left(4n+2,2,2\right)\right] & = & \theta_{\alpha\alpha}\left[ \left\{ n+1,n+2,\dots,2n\right\}\right]\\
& = &\{ \alpha_{1}, \alpha_{2}, \dots, \alpha_{n}\}\\
& = & P\left(4n+2,2,\phi\left(2\right)\right)\\
& = & P\left(4n+2,2,1\right)\end{eqnarray*}

\begin{eqnarray*}
\theta_{\alpha\alpha}\left[P\left(n,2,2\right)\right] & = & \theta_{\alpha\alpha}\left[ \left\{ n+1,n+2,\dots,2n\right\}\right]\\ & =& \left\{ 1,2,\dots,n\right\}\\ & =& P\left(10,2,\phi\left(2\right)\right)\\ &=& P\left(10,2,1\right)\end{eqnarray*}

and
\begin{eqnarray*}
\theta_{\alpha\alpha}^{2}
& = & \left(\begin{array}{cccccccc}
1 & 2 & \dots & n & n+1 & n+2 & \dots & 2n\\
{\alpha_{1}} & {\alpha_{2}} & \dots & {\alpha_{n}} & {{\alpha_{1}+n}} & {{\alpha_{2}+n}} & \dots & {{\alpha_{n}+n}}\end{array}\right)\\
& = & \left(\begin{array}{cccc}
1 & 2 & \dots & n\\
{\alpha_{1}} & {\alpha_{2}} & \dots & {\alpha_{n}}\end{array}\right)\mid\diamond
\left(\begin{array}{cccccccc}
1 & 2 & \dots & n\\
{\alpha_{1}} & {\alpha_{2}} & \dots & {\alpha_{n}}\end{array}\right)\\
& =& \alpha_{2n+1}\diamond\mid\alpha_{2n+1}\end{eqnarray*}.\medskip

\item$\theta_{4n+2}=Id_{2n+1}\mid\diamond\beta_{2n+1}$
The proof is analogous to {(3)}.

\end{enumerate}
\end{proof}

\begin{example}
The permutations $$\alpha_{3}\mid\diamond Id_{3}=\left(\begin{array}{cccccc}
1 & 2 & 3 & 4 & 5 & 6\\
6 & 4 & 5 & 1 & 2 & 3\end{array}\right),$$ $$\beta_{3}\mid\diamond Id_{3}=\left(\begin{array}{cccccc}
1 & 2 & 3 & 4 & 5 & 6\\
5 & 6 & 4 & 1 & 2 & 3\end{array}\right),$$ $$Id_{3}\mid\diamond \alpha_{3}=\left(\begin{array}{cccccc}
1 & 2 & 3 & 4 & 5 & 6\\
4 & 5 & 6 & 3 & 1 & 2\end{array}\right)$$ and $$Id_{3}\mid\diamond \beta_{3}=\left(\begin{array}{cccccc}
1 & 2 & 3 & 4 & 5 & 6\\
4 & 5 & 6 & 2 & 3 & 1\end{array}\right)$$ are simple permutations of order 6.
\end{example}

	\section{Final Remarks and Open Questions}

Taking into account that Pasting and Reversing operations have been used to describe the genealogy of simple permutations with order a power of two in a recursive way and the first particular case of mixed order $(4n+2)$ in a constructive way, the use of those operations as a way to describe periodic orbits can lead to a new perspective. The following problems are related to this paper and can be source of future papers.

The first problem to be developed from this paper, is the extension of the found theorems in order $4n+2$ to the following order $8n+4$ and subsequently to the complete ``middle tail" of Sharkovskii's Order. it is necessary to take into account that, in order $8n+4$ there are several permutations $\phi$ with order a power of two, and it would increase the ways to construct simple permutations with mixed order.

The second open problem to be developed is the relationship between the simple permutations and the dynamic of the associated primitive functions. Markov Graphs shows some facts about this dynamic (for example, the existence of unplugged vertices and its relationship with the existence of some periods, critical points, etc).

The third open problem suggested to the reader is concerning to Pasting and Reversing operations. These intuitive operations have been used successfully in several math contexts (rings, permutations, etc). According to this fact, the study of its properties and its definition in new algebraic structures has to be considered also as a research field.

\medskip
%% The data information below will be filled by AIMS editorial staff
%Received September 2006; revised February 2007.

\medskip


\begin{thebibliography}{99}

\bibitem{Genealogia}
\newblock P. Acosta-Hum\'anez,
\newblock \emph{Genealogy of simple permutations with order a power of two} (Spanish).
\newblock Revista Colombiana de Matem\'aticas, {\bf 2} (2008), 1--14.

\bibitem{Pegamiento}
\newblock P. Acosta-Hum\'anez,
\newblock \emph{Pasting operation and the square of natural numbers} (Spanish).
\newblock Civilizar, {\bf 2} (2003), 85--97.

\bibitem{preprint}
\newblock P. Acosta-Hum\'anez, A. Chuqu\'en, \'A. Rodr\'iguez,
\newblock \emph{On Pasting and Reversing operations over some rings,}
\newblock Bolet\'in de Matem\'aticas Universidad Nacional, Volume 17, Issue 2,(2010),
143--164.

\bibitem{Combinatorial}
\newblock LL. Alseda, J. Llibre and M. Misiurewicz,
\newblock \emph{Combinatorial dynamics and entropy in dimension one,}
\newblock Advanced Series in Nonlinear Dynamics, second edition, World Scientific publishing Vol 5
\textbf{15} (2005).

\bibitem{Minimal}
\newblock Ll, Alseda, J. Llibre, R. Serra
\newblock \emph{Minimal periodic orbits for continuous maps of the interval,}
\newblock Transactions of the American Mathematical Society, Volume 286, Issue 2, {\bf 2} (1977), 595--627.

\bibitem{Bernhardt}
\newblock C. Bernhardt,
\newblock \emph{Simple permutations with order a power of two,}
\newblock Ergodic Theory and Dynamic Systems, {\bf 2} (1984), 179--186.

\bibitem{Block1}
\newblock L. Block,
\newblock \emph{Dynamic in one dimension,}
\newblock Lecture Notes in mathematics, Springer Verlag, New York, {\bf 2} (1986).

\bibitem{Block2}
\newblock L. Block,
\newblock \emph{Simple periodic orbits or mappings of the interval,}
\newblock Transactions of the American Mathematical Society, Volume 254, {\bf 2} (1979), 391--398.

\bibitem{Ho}
\newblock C. Ho,
\newblock \emph{On the structure of minimum orbits of periodic points for maps on the real line,}
\newblock  Preprint.

\bibitem{Misiurewicz}
\newblock M. Misiurewicz,
\newblock \emph{Thirty years after Sharkovskii's theorem,}
\newblock Thirty years after Sharkovskii's theorem: new perspectives, Murcia {\bf 2} (1995), 13--20.
\bibitem{Sharkovskii}
\newblock A. Sharkovskii,
\newblock \emph{Coexistence of cycles of a continuous map of the line into itself,}
\newblock Ukrain Mat. Zh. {\bf 16} (1964), 61--71.

\bibitem{Sharkovskii2}
\newblock A. Sharkovskii,
\newblock \emph{Coexistence of cycles of a continuous map of the line into itself,}
\newblock Thirty years after Sharkovskii's theorem: new perspectives, Murcia {\bf 2} (1995), 1--12.

\bibitem{Stefan}
\newblock P. Stefan,
\newblock \emph{A theorem of Sharkovskii on the existence of periodic orbits of continuous endomorphisms of the real line,}
\newblock Communications in mathematical physics, Springer Verlag, New York {\bf 2} (1977), 237--248.

\end{thebibliography}
\end{document}